\documentclass[11pt]{article}
\usepackage{amssymb,amsfonts,amsmath,latexsym,epsf,tikz,url}

\newtheorem{theorem}{Theorem}[section]

\newtheorem{lemma}[theorem]{Lemma}
\newtheorem{remark}[theorem]{Remark}

\newcommand{\qed}{\hfill $\square$\medskip}

\begin{document}

\title{On the signed domination number of some Cayley graphs}

\author{
S. Alikhani$^{1,}$\footnote{Corresponding author}
\and F. Ramezani$^1$  \and E. Vatandoost$^2$ 
}

\date{\today}

\maketitle

\begin{center}
$^1$Department of Mathematics, Yazd University, 89195-741, Yazd, Iran\\
{\tt alikhani@yazd.ac.ir, f.ramezani@yazd.ac.ir}
\medskip 

$^2$Department of Basic Science, Imam Khomeini International University, Qazvin, Iran
{\tt  vatandoost@ikiu.ac.ir}
\end{center}


\begin{abstract}
A {\it signed
	dominating function} of graph $\Gamma$ is a function $g :
V(\Gamma) \longrightarrow \{-1,1\}$ such that $\sum_{u \in N[v]}g(u) >0$ for each $v \in
V(\Gamma)$. The {\it signed domination number} $\gamma_{_S}(\Gamma)$ is the
minimum weight of a signed dominating function on $\Gamma$.  Let $G=\langle S \rangle$ be a finite group such that $e \not\in S=S^{-1}$. 
In this paper, we obtain the signed domination number of $Cay(S:G)$ based on cardinality of $S$. Also we determine the classification of group $G$ by $|S|$ and $\gamma_{_S}(Cay(S:G))$.
\end{abstract}


\medskip
\noindent{\bf AMS Subj.\ Class.:} 05C50, 05C25

\section{Introduction}

Let $\Gamma$ be a simple graph on vertex set $V(\Gamma)$ and edge
set $E(\Gamma)$. The number of edges of the shortest walk joining $v$
and $u$ is called the {\it distance} between $v$ and $u$
and denoted by $d(v,u)$. The maximum value of the distance
function in a connected graph $\Gamma$ is called the {\it
	diameter} of $\Gamma$ and denoted by $diam(\Gamma)$.
A graph $\Gamma$ is said to be {\it regular} of degree $k$ or, $k${\it -regular} if every vertex has degree $k$.
A { \it $t-$partite} graph is one whose vertex set can be partitioned into $t$ subset, or part, in such a way that no edge has both end in a same part. A {\it complete $t-$partite} graph is a $t-$partite graph in which every pair of vertices in seprate parts are adjacent and denoted by $K_{n_1,\ldots,n_t}$ if $V=V_1 \cup \ldots \cup V_t$ and $|V_k|=n_k$.
Let $G$ be a non-trivial group, $S\subseteq G$, $S^{-1}=\{s^{-1}
| s\in S \}$, and also $e \notin S=S^{-1}$.
The {\it Cayley graph } of $G$ denoted by $Cay(S:G)$, is a graph with vertex set $G$ and two vertices $a$ and $b$ are adjacent if and only if $ab^{-1}\in S$.
If $S$ generates $G$, then $Cay(S:G)$ is connected. Also Cayley graph is a simple and vertex transitive.\\
For a  vertex $v \in V(\Gamma)$, the closed neighborhood $N[v]$ of
$v$ is the set consisting of $v$ and all of its neighbors. For a
function $g:V(\Gamma)\longrightarrow \{-1,1\}$ and a vertex $v
\in V$ we define $g[v]=\sum_{u \in N[v]} g(u)$.

 A {\it signed
	dominating function} of $\Gamma$ is a function $g :
V(\Gamma) \longrightarrow \{-1,1\}$ such that $g [v]>0$ for all $v \in
V(\Gamma)$. The {\it weight} of a function $g$ is $\omega
(g)=\sum_{v \in V(\Gamma)}
g(v)$. The {\it signed domination number} $\gamma_{_S}(\Gamma)$ is the
minimum weight of a signed dominating function on $\Gamma$.
 A
signed dominating function of weight $\gamma_{_S}(\Gamma)$ is called
a $\gamma_{_S}-$function. Also $V^-$ is the set of vertices which asigned $-1$  by a $\gamma_{_S}-$function and $V^-_g$ is $\{v\in \Gamma: \, g(v)=-1 \}$ where $g$ is a signed dominating function. The concept of signed domination number was defined in
\cite{dunbar} and has been studied by several authors (see for
instance \cite{dunbar,3,4,7,10,Zelinka}). In \cite{vatandoost} domination number as well as signed
domination numbers of $Cay(S:G)$ for all cyclic group $G$ of order $n$, where
$n \in \lbrace p^m, pq \rbrace$ and $S=\lbrace k<n \, : \, gcd(k,n)=1  \rbrace$ are investigated. 
\medskip 
Motivated by \cite{vatandoost}, in this paper we  determine the signed domination number of Cayley graphs $Cay(S:G)$ for such pairs of $G$ and $S$. Also we obtain the group $G$ based on  $\gamma_{_S}(Cay(S:G)).$

\section{Computation of  $\gamma_{_S}(Cay(S:G))$}
In this section, $G=\langle S \rangle$ is a finite group of order $n$ and $S$ denote an inverse closed subset of $G  \setminus \lbrace e \rbrace$. We determine the signed domination number of $Cay(S:G)$ where $n-4\le |S|\le n-1$. We need the following lemma and theorem:

\begin{lemma} \rm{\cite{7}}\label{kn} If $\Gamma$ is a complete graph of order $n$, then
	\begin{eqnarray*}\gamma_{_S}(\Gamma )= \left \{ \begin{array}{ll}
			$1$ &  \text{if} ~$n$ ~ \text{is odd},\\
			$2$ & \text{if} ~$n$~ \text{is even}.
			
		\end{array} \right
		.\end{eqnarray*}
\end{lemma}
\begin{theorem} \label{1} \rm{\cite{dunbar,hening}} Let $\Gamma$ be a $k-$regular graph of order $n$. If $k$ is odd, then $\gamma_{_S}(\Gamma )\geq \frac{2n}{k+1}$ and if $k$ is even, then $\gamma_{_S}(\Gamma )\geq  \frac{n}{k+1} .$
\end{theorem}

\begin{theorem} Let $G$ be a group of order $n$. Then $\gamma_{_S}(Cay(S:G))=1$ if and
	only if $S=G\setminus \{e \}$ and $n$ is odd.
\end{theorem}
\begin{proof} Let
	$\gamma_{_S}(Cay(S:G))=1$. We know that $Cay(S:G)$ is a $|S|-$regular graph. If
	$|S|$ is odd, then $\gamma_{_S}((Cay(S:G))\geq
	\frac{2n}{|S|+1}$ by Theorem \ref{1} and so $|S|\geq (2n-1)$. This
	is impossible. Hence, $|S|$ is even and by Theorem \ref{1},
	$\gamma_{_S}(Cay(S:G))\geq \frac{n}{|S|+1}$. So $|S|=n-1$ and $n$ is odd. The converse is clear.\qed
\end{proof}
\begin{theorem}\label{3} If $G$ is a group of order $n$ and  $|S|=n-2$, then $\gamma_{_S}(Cay(S:G))=2$.
\end{theorem}
\begin{proof}
	Since $G$ is $|S|-$regular, so by Theorem \ref{1}, $\gamma_{_S}(Cay(S:G))\geq \frac{n}{|S|+1}$ and so $|V^-|\leq \frac{n}{2}-1$. Let
	$S=\{v_1,\ldots,v_{n-2}\}$ and $G=S \cup \{e,a\}$.
	
	\noindent Define $f:V(Cay(S:G)) \rightarrow
	\{-1,1 \}$ such that
	\begin{eqnarray*} f(v_i)= \left \{
		\begin{array}{ll} $-1$     &
			i=1,\ldots,\frac{n}{2}-1
			,\\
			$1$      &  $otherwise$.
			
		\end{array} \right
		.\end{eqnarray*} 
	Since $deg(a)=|S|$ and $a \notin S$, so
	$N(a)=S$ and also $f[e]=f[a]=1$. For every $v_i \in S$ there is
	exactly one $v_j \in S$ such that $v_i$ and $v_j$ are not
	adjacent. Suppose that $f(v_i)=-f(v_j)=-1$, then $f[v_i]=1$. Hence,
	$f[v]\geq 1$ for every $v \in V(Cay(S:G))$. Since $|V^-_f|=\frac{n}{2}-1$, $f$ is a
	$\gamma_{_S}-$function and so $\gamma_{_S}(Cay(S:G))=\omega(f)=2$. \qed
\end{proof}

\begin{theorem}\label{5} Let $G$ be a group of order $n$ and $|S|=n-3$. Then
	\begin{eqnarray*}
		\gamma_{_S}(Cay(S:G))= \left \{
		\begin{array}{ll} 3     &
			\text{if} ~$n$ ~ \text{is odd},\\
			4      &  \text{if} ~$n$ ~ \text{is even}.
			
		\end{array} \right
		.\end{eqnarray*}
	
\end{theorem}
\begin{proof} Let $S=\{v_1,\ldots,v_{n-3}\}$ and $G \setminus S=\{e,x,y \}$. If $n$ is odd,
	then $|S|=n-3$ is even. By Theorem \ref{1},
	$\gamma_{_S}(Cay(S:G))\geq\frac{n}{|S|+1}$ and so $|V^-|\leq
	\frac{n-3}{2}$. It is sufficient to define a signed dominating function $g$ such that $V^-_g \subset S$ and $|V^-_g|=\frac{n-3}{2}$. Hence, for each $x \in G$, $g[x]=|S|+1-2|N[x]\cap V^-_g|\ge n-2-|V^-_g|=1$. Thus $$\gamma_{_S}(Cay(S:G))=\omega(g)=3$$
	Now, if $n$ is even, then $|S|$ is odd and so by
	Theorem \ref{1}, $\gamma_{_S}(Cay(S:G))\geq\frac{2n}{n-2}$. Hence,
	$|V^-|\leq \frac{n-4}{2}$ and so each signed dominating function that gives label $-1$ to exactly $\frac{n-4}{2}$ vertices  is a $\gamma_{_S}$-function. Thus $\gamma_{_S}(Cay(S:G))=4.$ \qed
\end{proof}

\begin{lemma}\label{K3} Let $G$ be a group of order $n$ and $S=G \setminus \lbrace e,a,b,c \rbrace$. Then $n$ is even and the induced subgraph on $\lbrace a,b,c \rbrace$ in $Cay(S:G)$ is $K_3$, $P_3$ or empty graph.
\end{lemma}
\begin{proof}
	Since $S$ is an inverse closed subset of $G$, at least one of vertices $a,b$ or $c$ as an element of group $G$ has order two and so $n$ is even. Let $A=\lbrace a,b,c \rbrace $ and consider two following cases:
	\begin{itemize}
		\item[Case 1:]  Let $O(x)=2$ for each $x \in A$. If $a \in N(b)$, then $ab \in S$ and so $ab \neq c$. Hence, $cb \neq a$ and $ac\neq b$ and so $bc,ac \in E(Cay(S:G))$. Therefore, the induced subgraph on $A$ is $K_3$. If $a \not \in N(b) $, then $ab=c$ and so $cb=a$ and $ac=b$. Thus the induced subgraph on $A$ has no edge.
		\item[Case 2:] Let $O(a)=2$ and $c=b^{-1}$. There are two cases.  Suppose that $4 \mid n$ and $O(b)=4$. If $a = b^2$, then the induced subgraph on $A$ is empty. If $a \neq b^2$, then $c^2=b^2 \in S$. So $ba,ca,bc \in S$. Hence, the induced subgraph on $A$ is $K_3.$ If $4 \nmid n$ or $O(b) \neq 4$, then $ab,ac \not \in A \cup \lbrace e \rbrace$. So $a \in N(b) \cap N(c)$. If $O(b)=3$, then $bc \not \in E(Cay(S:G))$ and so the induced subgraph on $A$ is $P_3$. Otherwise the induced subgraph on $A$ is $K_3$. \qed
	\end{itemize}
\end{proof}

\begin{lemma}\rm{\cite{vatandoost}} \label{Kbakhsh}
	 Let $G$ be a group and $H\neq G$ be a subgroup of $G$ such that $[G:H]=t$. If $S=G \setminus H$, then $Cay(S:G)$ is a complete $t$-partite graph.
\end{lemma}

\begin{theorem} \label{6}
	 Let $G=\langle S\rangle$ be  a finite group of order $n$ and $e \notin S=S^{-1}$ and $|S|=n-4$. Then $\gamma_{_S}(Cay(S:G))=4$.
\end{theorem}
\begin{proof}
	 Let $G=S \cup \lbrace e,a,b,c\rbrace$. Since $S$ is an inverse closed subset of $G$, at least one of elements $a,b$ or $c$ has order two, so both of $n$ and $|S|$ are even. By Theorem
	\ref{1}, $\gamma_{_S}(Cay(S:G)) \geq \frac{n}{n-3}$. So $|V^-|\leq \frac{n}{2}-1$. If
	$|V^-|=\frac{n}{2}-1$, then there is vertex $u$ such that
	$f[u]=-1$. This is contradiction. Hence, $|V^-|\leq
	\frac{n}{2}-2$. Let $S=\lbrace v_1,\ldots,v_{n-4}\rbrace$. Define
	$f:V(Cay(S:G))\rightarrow \{-1,1 \}$ such that $f(v_i)=-1$ if and only if $1 \le i \le \frac{n}{2}-2.$ Since
	$|V^-_f|=\frac{n}{2}-1$, so $f$ is a $\gamma_{_S}-$function. Therefore,
	$\gamma_{_S}(Cay(S:G))=\omega(f)=4$. \qed
\end{proof}
\section{Determining the group $G$ based on $\gamma_{_S}(Cay(S:G))$} 
Let $\Gamma_G$ be the Cayley graph $Cay(S:G)$ where $S$ is an inverse closed subset of $G  \setminus \lbrace e \rbrace$.
In this section we determine the finite group $G$ based on $\gamma_{_S}(\Gamma_G).$ In  the following there are some remarks for characterizing  all cubic Cayley graphs where $n \in \{8,10,12\}.$

\begin{remark} Let $G$ be a group of order $n$. By Lemma 1 of \rm{\cite{7}}, $\gamma_{_S}(\Gamma_G)=n$ if and only if $G=\mathbb{Z}_2$.
	Also for graph $\Gamma$ of order  $n$, $\gamma_{_S}(\Gamma)\neq n-t$, where $t$ is odd.
\end{remark}

\begin{remark} \label{D8} 
	Let $\Gamma_1$ and $\Gamma_2$ be graphs in Figure \ref{fig}. If $|S|=3$ and $O(s)=2$ for every $s \in S$, then $Cay(S:D_8)\simeq\Gamma_1$ or $\Gamma_2$ and $Cay(S:\mathbb{Z}_2 \times \mathbb{Z}_2 \times \mathbb{Z}_2)\simeq\Gamma_2$. 
\end{remark}

\begin{remark}\label{D8Z2*Z4} 
	The Cayley  graph of groups  $D_8$ and $\mathbb{Z}_2 \times \mathbb{Z}_4$ with $S=\{s_1,s_2,s_1^{-1}\}$, is isomorphic to $\Gamma_1$  and $Cay\left(\{s_1,s_2,s_1^{-1}\}:\mathbb{Z}_8 \right) \simeq \Gamma_2$ in Figure \ref{fig}.
\end{remark}

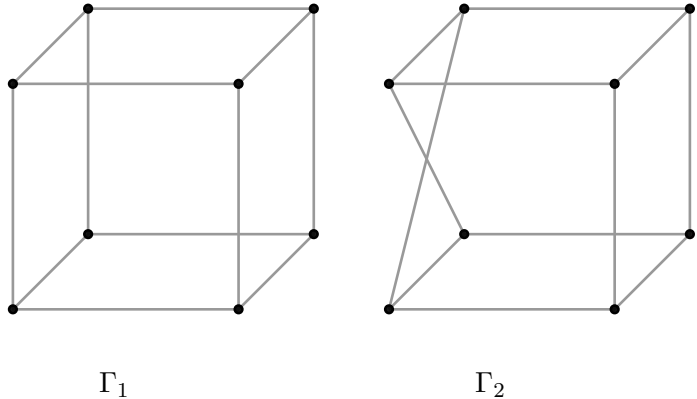
\begin{figure}[h]
	\centering
	\begin{tikzpicture}
	[line width=1pt,scale=1]
	\draw[black!40](0,2)--(3,2);
	\draw[black!40](2,1)--(3,2);
	\draw[black!40](2,1)--(-1,1);
	\draw[black!40](0,2)--(-1,1);
	\draw[black!40](3,2)--(3,-1)--(0,-1);
	\draw[black!40](2,1)--(2,-2)--(3,-1);
	\draw[black!40](-1,1)--(-1,-2)--(2,-2);
	\draw[black!40](0,2)--(0,-1)--(-1,-2);

	\draw[black!40](5,2)--(8,2);
	\draw[black!40](7,1)--(8,2);
	\draw[black!40](7,1)--(4,1);
	\draw[black!40](5,2)--(4,1);
	\draw[black!40](8,2)--(8,-1)--(5,-1);
	\draw[black!40](7,1)--(7,-2)--(8,-1);
	\draw[black!40](4,-2)--(7,-2);
	\draw[black!40](5,2)--(4,-2);
	\draw[black!40](4,1)--(5,-1)--(4,-2);
	
	\filldraw[fill opacity=0.9,fill=black]  (5,-1) circle (0.05cm);
	\filldraw[fill opacity=0.9,fill=black]  (4,-2) circle (0.05cm);
	\filldraw[fill opacity=0.9,fill=black]  (7,-2) circle (0.05cm);
	\filldraw[fill opacity=0.9,fill=black]  (7,1) circle (0.05cm);
	\filldraw[fill opacity=0.9,fill=black]  (5,2) circle (0.05cm);
	\filldraw[fill opacity=0.9,fill=black]  (8,2) circle (0.05cm);
	\filldraw[fill opacity=0.9,fill=black]  (8,-1) circle (0.05cm);
	\filldraw[fill opacity=0.9,fill=black]  (4,1) circle (0.05cm);
	
	\filldraw[fill opacity=0.9,fill=black]  (0,-1) circle (0.05cm);
	\filldraw[fill opacity=0.9,fill=black]  (-1,-2) circle (0.05cm);
	\filldraw[fill opacity=0.9,fill=black]  (2,-2) circle (0.05cm);
	\filldraw[fill opacity=0.9,fill=black]  (2,1) circle (0.05cm);
	\filldraw[fill opacity=0.9,fill=black]  (0,2) circle (0.05cm);
	\filldraw[fill opacity=0.9,fill=black]  (3,2) circle (0.05cm);
	\filldraw[fill opacity=0.9,fill=black]  (3,-1) circle (0.05cm);
	\filldraw[fill opacity=0.9,fill=black]  (-1,1) circle (0.05cm);
	
	\coordinate [label=right:{$\Gamma_1$}] (g) at (0,-3);
	\coordinate [label=right:{$\Gamma_2$}] (g) at (5,-3);
	\end{tikzpicture}
	\caption{The cubic Cayley graphs of order $8.$} 
	\label{fig}
\end{figure}
\begin{remark}\label{D10Z10} Any cubic Cayley graphs of order $10$ is isomorphic to one of the graphs in Figure \ref{FigureCay(S:D_10)}.
\end{remark}

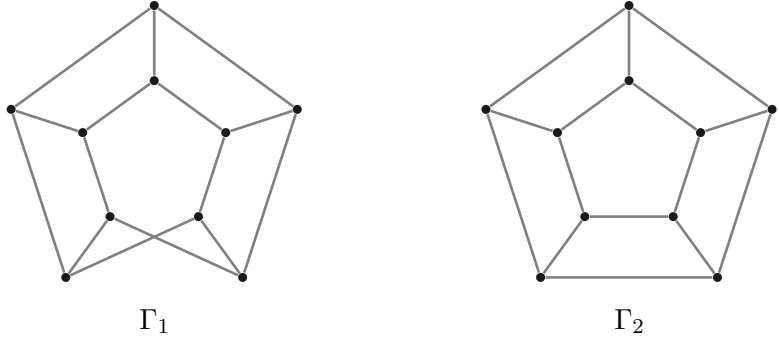
\begin{figure}[h]
	\centering
	\begin{tabular}{cccccc}
		\begin{tikzpicture}[
		vertex_style/.style={circle,inner sep=0pt,minimum size=0.12cm, fill opacity=0.9, fill=black},
		edge_style/.style={line width=1pt, gray}]
		
		
		\begin{scope}[rotate=90]
		\foreach \x/\y in {0/1,72/2,144/3,216/4,288/5}{
			\node[vertex_style] (\y) at (canvas polar cs: radius=1cm,angle=\x){};
		}
		\foreach \x/\y in {0/6,72/7,144/8,216/9,288/10}{
			\node[vertex_style] (\y) at (canvas polar cs: radius=2cm,angle=\x){};
		}
		\end{scope}
		
		\foreach \x/\y in {1/6,2/7,3/8,4/9,5/10}{
			\draw[edge_style] (\x) -- (\y);
		}
		
		\foreach \x/\y in {1/2,2/3,4/5,5/1,4/8}{
			\draw[edge_style] (\x) -- (\y);
		}
		
		\foreach \x/\y in {6/7,7/8,3/9,9/10,10/6}{
			\draw[edge_style] (\x) -- (\y);
		}
		\coordinate [label=right:{$\Gamma_1$}] (h) at (-0.33,-2.2);
		
		\end{tikzpicture}  & & \hspace*{0.7cm} & &
		
		\begin{tikzpicture}[
		vertex_style/.style={circle,inner sep=0pt,minimum size=0.12cm, fill opacity=0.9, fill=black},
		edge_style/.style={line width=1pt, gray}]
		
		
		\begin{scope}[rotate=90]
		\foreach \x/\y in {0/1,72/2,144/3,216/4,288/5}{
			\node[vertex_style] (\y) at (canvas polar cs: radius=1cm,angle=\x){};
		}
		\foreach \x/\y in {0/6,72/7,144/8,216/9,288/10}{
			\node[vertex_style] (\y) at (canvas polar cs: radius=2cm,angle=\x){};
		}
		\end{scope}
		
		\foreach \x/\y in {1/6,2/7,3/8,4/9,5/10}{
			\draw[edge_style] (\x) -- (\y);
		}
		
		\foreach \x/\y in {1/2,2/3,3/4,4/5,5/1}{
			\draw[edge_style] (\x) -- (\y);
		}
		
		\foreach \x/\y in {6/7,7/8,8/9,9/10,10/6}{
			\draw[edge_style] (\x) -- (\y);
		}
		\coordinate [label=right:{$\Gamma_2$}] (h) at (-0.33,-2.2);
		\end{tikzpicture}
		
	\end{tabular}
	\caption{The cubic Cayley graphs of order $10.$}
	\label{FigureCay(S:D_10)}
\end{figure}
\begin{remark}\label{D12,Z12}
	Any cubic Cayley graph of groups $D_{12},\mathbb{Z}_{12}$ and $\mathbb{Z}_2  \times \mathbb{Z}_6$ is isomorphic to one of the graphs in Figure \ref{D12Z12}. Moreover let $S=\{s_1,s_2,s_1^{-1}\}.$ If $G\simeq D_{12} , \mathbb{Z}_2 \times \mathbb{Z}_6$, then $Cay(S:G)\simeq \Gamma_2$, otherwise  $Cay(S:\mathbb{Z}_{12})\simeq \Gamma_1.$
\end{remark}
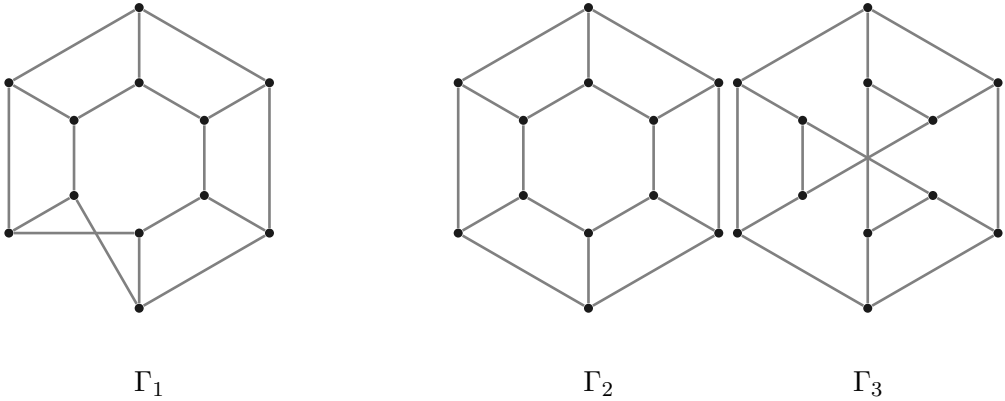
\begin{figure}[h]
	\centering
	\begin{tabular}{cccccc}

		\begin{tikzpicture}[
		vertex_style/.style={circle,inner sep=0pt,minimum size=0.12cm, fill opacity=0.9, fill=black},
		edge_style/.style={line width=1pt, gray}]
		
		
		\begin{scope}[rotate=90]
		\foreach \x/\y in {0/1,60/2,120/3,180/4,240/5,300/6}{
			\node[vertex_style] (\y) at (canvas polar cs: radius=1cm,angle=\x){};
		}
		\foreach \x/\y in {0/7,60/8,120/9,180/10,240/11,300/12}{
			\node[vertex_style] (\y) at (canvas polar cs: radius=2cm,angle=\x){};
		}
		\end{scope}
		
		\foreach \x/\y in {1/7,2/8,3/9,4/10,5/11,12/6}{
			\draw[edge_style] (\x) -- (\y);
		}
		
		\foreach \x/\y in {1/2,2/3,3/10,4/5,5/6,6/1}{
			\draw[edge_style] (\x) -- (\y);
		}
		
		\foreach \x/\y in {11/12,7/8,8/9,9/4,10/11,12/7}{
			\draw[edge_style] (\x) -- (\y);
		}
		\coordinate [label=right:{$\Gamma_1$}] (h) at (-0.2,-3);
		\end{tikzpicture} & & \hspace*{0.7cm} & & 
		
		\begin{tikzpicture}[
		vertex_style/.style={circle,inner sep=0pt,minimum size=0.12cm, fill opacity=0.9, fill=black},
		edge_style/.style={line width=1pt, gray}]
		
		
		\begin{scope}[rotate=90]
		\foreach \x/\y in {0/1,60/2,120/3,180/4,240/5,300/6}{
			\node[vertex_style] (\y) at (canvas polar cs: radius=1cm,angle=\x){};
		}
		\foreach \x/\y in {0/7,60/8,120/9,180/10,240/11,300/12}{
			\node[vertex_style] (\y) at (canvas polar cs: radius=2cm,angle=\x){};
		}
		\end{scope}
		
		\foreach \x/\y in {1/7,2/8,3/9,4/10,5/11,12/6}{
			\draw[edge_style] (\x) -- (\y);
		}
		
		\foreach \x/\y in {2/3,2/1,3/4,5/4,6/5,6/1}{
			\draw[edge_style] (\x) -- (\y);
		}
		
		\foreach \x/\y in {11/12,7/8,8/9,9/10,10/11,12/7}{
			\draw[edge_style] (\x) -- (\y);
		}
		\coordinate [label=right:{$\Gamma_2$}] (h) at (-0.2,-3);
		\end{tikzpicture}
		\begin{tikzpicture}[
		vertex_style/.style={circle,inner sep=0pt,minimum size=0.12cm, fill opacity=0.9, fill=black},
		edge_style/.style={line width=1pt, gray}]
		
		
		\begin{scope}[rotate=90]
		\foreach \x/\y in {0/1,60/2,120/3,180/4,240/5,300/6}{
			\node[vertex_style] (\y) at (canvas polar cs: radius=1cm,angle=\x){};
		}
		\foreach \x/\y in {0/7,60/8,120/9,180/10,240/11,300/12}{
			\node[vertex_style] (\y) at (canvas polar cs: radius=2cm,angle=\x){};
		}
		\end{scope}
		
		\foreach \x/\y in {1/7,2/8,3/9,4/10,5/11,12/6}{
			\draw[edge_style] (\x) -- (\y);
		}
		
		\foreach \x/\y in {2/3,2/5,3/6,1/4,4/5,6/1}{
			\draw[edge_style] (\x) -- (\y);
		}
		
		\foreach \x/\y in {11/12,7/8,8/9,9/10,10/11,12/7}{
			\draw[edge_style] (\x) -- (\y);
		}
		\coordinate [label=right:{$\Gamma_3$}] (h) at (-0.33,-3);
		
		\end{tikzpicture}
	\end{tabular}
	\caption{The cubic Cayley graphs where $G \in \{ D_{12}\, , \, \mathbb{Z}_{12}\, , \, \mathbb{Z}_2 \times \mathbb{Z}_6 \}.$}
	\label{D12Z12}
\end{figure}

\begin{theorem} \label{Cn}\rm{\cite{Haynes}} 
	 For $n \geq 3$, $\gamma_{_S}(C_n)=n-2\lfloor \frac{n}{3} \rfloor $.
\end{theorem}
\begin{theorem}\label{d3} \rm{\cite{9}} 
	Let $\Gamma$ be a graph with $\Delta \leq 3$, $g$ be a signed dominating function of $\Gamma$ and $u,v \in V(\Gamma)$. If $g(u)=g(v)=-1$, then $d(u,v) \geq 3$.
\end{theorem}

\begin{theorem}\label{n-2}
	 Let $G$ be a group of order $n.$ If $\gamma_{_S}(\Gamma_G)=n-2$, then $G\simeq \mathbb{Z}_3,\mathbb{Z}_2 \times \mathbb{Z}_2,\mathbb{Z}_4,\mathbb{Z}_5,S_3,\mathbb{Z}_6, \mathbb{Z}_8$ or $D_8$.
\end{theorem}
\begin{proof}
	Since $\gamma_{_S}(\Gamma_G)=n-2$, so $|V^-|=1$. Let $f$ be a $\gamma_{_S}$-function. Since Cayley graph is a vertex transitive graph, we can suppose that $f(e)=-1$ and $f(x)=1$ for all $x \in G\setminus \{e\}$. If $|S|=2$, then $\Gamma$ is a cycle. By Theorem \ref {Cn}, $n \in \{3,4,5\}$. So $G \simeq \mathbb{Z}_3,\mathbb{Z}_2 \times \mathbb{Z}_2,\mathbb{Z}_4,\mathbb{Z}_5.$\\
	If $|S|\geq 4$, then $n\geq 5$. If the induced subgraph on $S$ is
	$K_{n-1}$, then $\Gamma_G\simeq K_n$. Thus $\gamma_{_S}(K_n)=1$ or $2$ and so
	$n=3$ or $4$, respectively, which is impossible. So the induced subgraph on $S$ is not a complete graph and so there are $s,s'\in S$ such that $ss' \not\in E(\Gamma_G)$. Since $\Gamma_G$ is $|S|$-regular, there is $t \in N(S)\setminus (S\cup \{e\}).$ Also the lable of $t$ can be $-1$ (because $|S| \ge 4$). Thus $|V^-_f|\ge 2$. This is contradiction by $\gamma_{_S}(\Gamma_G)=n-2$. Therefore, $|S|=3$ and $n$ is even. Since $\gamma_{_S}(\Gamma_G)=n-2$, by Theorem \ref{d3}, $diam(\Gamma_G) \leq 2$. Let $S=\{s_1,s_2,s_3 \}$ and consider two general cases:
	\begin{itemize}
		\item[Case 1:] Let $O(s_i)=2$ for every $1 \leq i \leq 3$. Suppose that the induces subgraph on $S$ has the edge $s_1s_2$. Hence, $s_1s_2=s_3$ and so $s_1s_3=s_2 $ and $s_3s_2=s_1$. Thus  $\langle S \rangle \simeq K_3$, so $\Gamma_G\simeq K_4$ and so $G\simeq \mathbb{Z}_2 \times \mathbb{Z}_2$.
		If $\langle S \rangle$ is isomorphic to an empty graph, then $diam(\Gamma_G)=2$ and also $N(s_i)=\{e,x_{i},y_{i} \}$ for $1\leq i \leq 3.$ If $N(s_1)=N(s_2)=N(s_3)$, then
		$|V(\Gamma_G)|=6$. Since $\mathbb{Z}_6$ dose not have three elements of order two, $G \not\simeq \mathbb{Z}_6$. Let $G\simeq S_3$ and $S=\{(1\, 2),(1\, 3),(2\, 3) \}$. Then $\gamma_{_S}(Cay(S:S_3))=n-2.$ Now let $N(s_1)=N(s_2)$ and $N(s_1)\cap N(s_3)=\lbrace e \rbrace$.
		Let $x_{3} \not \in N(y_{3})$. So there is a vertex $b$ such that
		$b \in N(x_{3}) \setminus S \cup N(S)$. So $f(b)$ can be $-1$.
		This is impossible. Hence, $x_{3}y_{3} \in E(\Gamma_G)$ and so
		$n=8$. Since $G=\langle S \rangle$ and $O(s_i)=2$, $G\not\simeq \mathbb{Z}_8, Q_8$ and $\mathbb{Z}_2 \times \mathbb{Z}_4$. Also $\Gamma_G \not\simeq  \Gamma_1 , \Gamma_2$ in Remark \ref{D8},  and so $G\not \in \lbrace D_8,\mathbb{Z}_2 \times \mathbb{Z}_2 \times \mathbb{Z}_2 \rbrace.$ Since $\Gamma_G$ is a cubic graph and $ diam(\Gamma_G)=2$, so  it is impossible that $N(s_1)=N(s_2)$ and $N(s_1)\cap N(s_3)=\{ e,x_1
		\}$. If $N(s_i)\neq N(s_j)$ for $1\leq i,j \leq 3$,
		then there are five cases:\\
		\begin{itemize}
			\item[(i)] If $N(s_1)\cap N(s_2)=\{ e,x \}$ and $N(s_2)\cap N(s_3)=\{ e,y \}$, then there are $z \in N(s_1)$ and $w \in N(s_3)$.
			Since $diam(\Gamma_G)=2$, so $n=8$. By Remark \ref{D8}, $\Gamma_G\simeq\Gamma_1$ and so $G\simeq D_8.$
			\item[(ii)] If $N(s_1)\cap N(s_2)\cap N(s_3)=\{ e,x\}$, then $N(s_i)=\{e,x,y_i \}$ for $1 \leq i \leq 3$. So $y_is_i \neq s_i$ and so $y_1s_1=s_2$, $y_2s_2=s_3$ and $y_3s_3=s_1$. Thus $y_1=s_2s_1$, $y_2=s_3s_2$ and $y_3=s_1s_3$. Also $xs_1=s_3$, $xs_2=s_1$ and $xs_3=s_2$. Hence, $s_3s_1=s_1s_2=s_2s_3$. On the other hand, since $diam(\Gamma_G)=2$, so $n=8$. Thus $\langle S\rangle \simeq  K_3$. But we have $y_3y_2^{-1}=s_1s_3(s_3s_2)^{-1}=e \not\in S$. This is impossible.
			\item[(iii)] Let $N(s_i)\cap N(s_j)=\lbrace e \rbrace$ for
			$i \neq j$. Since $\gamma_{_S}(\Gamma_G)=n-2$, the induced
			subgraph on $N(S)\setminus\{e \}$ is $ 2-$regular graph. So
			$n=10$. But group $\mathbb{Z}_{10}$ does not have three elements of order
			two. Let $G\simeq D_{10}=\langle \{a,b: a^2=b^5=(ab)^2=e    \}\rangle.$  We know that $\{a,ab,ab^2,ab^3,ab^4 \}$ are all elements of $D_{10}$ of order two. So $N(S)\setminus \{e\}=\{ b^j: 1 \le j \le 4\}.$ Hence, $Cay(S:D_{10}) \not\simeq \Gamma_G$ and so $G\not\simeq D_{10}$.
			\item[(iv)] If $N(s_1)\cap N(s_2)=\{e,x \}$ and $N(s_1)\cap N(s_3)= N(s_2)\cap N(s_3)=\lbrace e \rbrace$. Since $diam(\Gamma_G)=2$, $n=9$. This is impossible to have a cubic graph of order $9$.
			\item[(v)] If $N(s_1)\cap N(s_2)=\{e,x \}$, $N(s_1)\cap
			N(s_3)=\{e,y \}$ and $N(s_2)\cap N(s_3)=\{e,z\}$. Since $\Gamma_G$ is a
			cubic graph, so there exists $t \in G \setminus N[S]$ and so $d(e,t)=3$. This is contradiction by $diam(\Gamma_G)=2.$
		\end{itemize}
		\item[Case 2:] Let $s_1=s^{-1}_3$ and $O(s_2)=2.$ If $s_1s_2 \in E(\Gamma_G)$, then $s_1s_2=s_3$.
		So $s^2_1=s_2$. Thus $O(s_1)=4$. Hence, $G\simeq \mathbb{Z}_4$. The similar argument applies when $s_3s_2 \in E(\Gamma_G)$.  Let $s_1,s_3 \not \in N(s_2).$ Then $N(s_2)=\{ e,x_2,y_2\}$.  If $O(s_1)=3$, then $s_3=s^2_1$ and $s_1 \in N(s_3)$. If $\langle \{ s_1,s_3,x_2,y_3\}\rangle$, then $n=6$. Let $G\simeq \mathbb{Z}_6$ and $S=\{a^2,a^3,a^4 \}$ or $G\simeq S_3$ and $S=\{(1\, 2),(1\, 2\, 3),(1\, 3\, 2) \}$. Then the result holds.\\
		If  $x_i \in N(s_i)\setminus N(s_2)$ for $i=1,3$, then $x_1x_2,x_3y_2 \in E(\Gamma_G)$. So $|G|=8$.
		Since $O(s_1)=3$, this is impossible and if $x_1 \in N(s_1)\setminus N(s_2)$ and $s_3y_2 \in E(\Gamma_G)$, then this graph is not cubic.
		Now, let $O(s_1)\neq 3$. then  $\langle S \rangle$ is an empty graph. The argument is likewise
		Case 1 when $\langle S \rangle \simeq \emptyset$. If $N(s_1)=N(s_2)=N(s_3)$, then $n=6$. For $G\simeq \mathbb{Z}_6$ and $S=\{ \bar{1}, \bar{3}, \bar{5}\}$ the result holds. Suppose that $N(s_1)=N(s_2)$, $N(s_1)\cap N(s_3)=\{e\}$. Hence, $n=8.$ But all elements of $\mathbb{Z}_2 \times \mathbb{Z}_2 \times \mathbb{Z}_2$ have order two. Also $Q_8$ does not have any generator of order $3$ and for each generator of $\mathbb{Z}_8 , D_8$ or $ \mathbb{Z}_2 \times \mathbb{Z}_4$ of cardinality $3$, $Cay(S:G) \not\simeq \Gamma_G$. 
		
		Likewise Case 1, we can see $N(s_1)\cap N(s_3)=\{ e, x_1\}$ and $N(s_1)=N(s_2)$ and Cases (iv) and (v) are not happened. To complete the  proof it is sufficient to consider Case (i) to (iii):
		\begin{itemize}
			\item[(i)] Let $i \neq j$ and $N(s_i)\neq N(s_j)$. Then $Cay(\{a,a^3,b\},D_8)$ is isomorphic to $\Gamma_G$.
			\item[(ii)] If $N(s_1) \cap N(s_2)\cap N(s_3)=\{e,x\}$, then $xs_1 \in \{s_2,s_1^{-1}\}$, $xs_2 \in \{s_1,s_1^{-1}\}$ and $xs_1^{-1}\in \{s_1,s_2\}$. It is not hard to see that all of them are impossible.
			\item[(iii)] Let $N(s_i)\cap N(s_j)=\lbrace e \rbrace$ for $i \neq j$. Then $G \in \{\mathbb{Z}_{10}, D_{10}\}$. But $Cay(S:G) \not\simeq \Gamma_G$ when $|S|=3$.
			\qed
		\end{itemize}
	\end{itemize}
\end{proof}

\begin{theorem} Let $G\simeq \langle S \rangle$ be a group of order $n$ and $|S|=2$. Then
	$\gamma_{_S}(Cay(S:G))=n-4$ if and only if $G\simeq \mathbb{Z}_6,\mathbb{Z}_7,\mathbb{Z}_8,S_3$ and $D_8$.
\end{theorem}
\begin{proof} By assumption $|S|=2$ and so the graph $Cay(S:G)$ is $2$-regular and so is isomorphic to $C_n$.
	Let $\gamma_{_S}(Cay(S:G))=n-4$. By Theorem \ref{Cn}, $6\le n \le 8$. Since $S$ is an inverse closed subset of $G$ of order two,  $G \not \in \lbrace \mathbb{Z}_2 \times \mathbb{Z}_2 \times \mathbb{Z}_2, \mathbb{Z}_2 \times \mathbb{Z}_4,Q_8 \rbrace$ but the result is reached for other groups with following generators:
	$\mathbb{Z}_6=\langle \bar{1}, \bar{5}\rangle$, $S_3=\langle (1 \,2),(1\,3)\rangle$, $\mathbb{Z}_7=\langle \bar{1}, \bar{6}\rangle$, $\mathbb{Z}_8=\langle \bar{1}, \bar{7}\rangle$ and $D_8=\langle a^2b, a^3b\rangle$. \qed
\end{proof}

\begin{theorem}\label{n-4} 
	Let $G=\langle S \rangle$ be a group of order $n$ where $e \not \in S=S^{-1}$ and $|S|=3$. Then $\gamma_{_S}(Cay(S:G))=n-4$ if and only if $G\simeq D_8, \mathbb{Z}_2\times \mathbb{Z}_4, \mathbb{Z}_{10}, D_{10}, \mathbb{Z}_{12}, D_{12}, \mathbb{Z}_2\times \mathbb{Z}_6$ and $ A_4$.
\end{theorem}
\begin{proof} 
	Let $\Gamma_G$ be a cubic Cayley graph and $\gamma_{_S}(\Gamma_G)=n-4$. Then $n$ is even and $|V^-_f|=2$, where $f$ is a $\gamma_{_S}-$function. By Theorem \ref{1}, $n \geq 8$. Let $S=\lbrace s_1,s_2,s_3\rbrace$. Since Cayley graph is a vertex transitive graph, so we can assume that $f(e)=-1$. So $f(s_i)=f(x)=1$ for every $i \in \lbrace 1,2,3\rbrace$ and every $x \in N[S] \setminus \lbrace e \rbrace$. Hence, there is vertex $a \in G \setminus N[S]$ such that $f(a)=-1$. Now there are two general cases:
	\begin{itemize}
		\item[Case 1.] If $O(s_i)=2$ for every $1 \leq i \leq 3$,
		then the induced subgraph on $S$ is $K_3$ or empty. If the induced subgraph on $S$ is $K_3$, then $Cay(S:G)=K_4$. This is contradiction by $n \ge 8$, so the induced subgraph on $S$ has no edges.
		
		If $N(s_1)=N(s_2)=\{e,x,y\}$, then $x=s_3s_2=s_2s_1$ and $y=s_1s_2=s_3s_1$. It implies that $s_3=s_1s_2s_1=s_2s_1s_2$. If $x'\in N(s_3)$, then $x's_3\in \{s_1s_2\}$ and so $x' \in \{x,y\}$. Hence, $N(s_1)=N(s_2)=N(s_3)$, i.e. $n=6$. This is impossible. Now suppose that $N(s_i)\neq N(s_j)$ for $1 \le i\neq j \le 3$. Let $x \in \bigcap^{3}_{i=1} N(s_i)$. Then $x=s_1s_3=s_3s_2=s_2s_1$. If $y \in N(s_1), z \in N(s_2)$, then $y=s_3s_1=x^{-1}=s_1s_2=z$. So again $n=6$ which is impossible. Now we consider the following 4 cases:
		\begin{itemize}
			\item[(i)] Let $N(s_1) \cap N(s_2)=\lbrace e,x \rbrace$, $N(s_2) \cap N(s_3)=\lbrace e,y\rbrace$ and $N(s_1) \cap N(s_3)=\{e\}$. Then $|N[S]|=8$. Since there is $a \in  G \setminus N[S]$ such that $f(a)=-1$, so $n \ge 10$. If $n \ge 14$, then $|V^-|\ge 3$. Hence, $n \in \{10,12\}$. The groups $\mathbb{Z}_{10}, \mathbb{Z}_{12}, \mathbb{Z}_2 \times \mathbb{Z}_6, A_4$ and $T=\langle \{a,b: a^4=b^3=e, \, a^{-1}ba=b^{-1} \}\rangle$ do not have such this generator. But $Cay(\{a,ab^2,ab^4\} , D_{10}), Cay(\{a,ab,b^3\}, D_{12}) \simeq \Gamma_G$
			\item[(ii)] If $N(s_i) \cap N(s_j)=\lbrace e\rbrace$ for every $i,j \in \lbrace 1,2,3 \rbrace$, then $|N[S]|=10$. Since $n$ is even and $|V^-|=2$, so $n=12$. But for all groups of this order $Cay(S:G) \not\simeq \Gamma_G$.
			\item[(iii)] Let $N(s_1) \cap N(s_2)= \lbrace e,x \rbrace$ and $N(s_2)\cap N(s_3)=N(s_1)\cap N(s_3)=\lbrace e \rbrace$.  Likewise case (i), $n \in \{10,12\}$. For any inverse closed subset of $D_{10}, \mathbb{Z}_{10}$ of cardinality 3, $Cay(S:D_{10}), Cay(S:\mathbb{Z}_{10})\not\simeq \Gamma_G$ but $Cay(\{a,ab,ab^3\}, D_{12})\simeq \Gamma_G$. 
			\item[(iv)] Let $N(s_1) \cap N(s_2)=\lbrace e,x \rbrace$ , $N(s_1) \cap N(s_3)=\lbrace e,y \rbrace$ and $N(s_2) \cap N(s_3)= \lbrace e,z\rbrace$. Then $|N[S]|=7$. Thus $n \in \{8,10,12\}.$ By similar argument in Case (i), $n \neq 10,12$. The groups $Q_8, \mathbb{Z}_8$ and $ \mathbb{Z}_2 \times \mathbb{Z}_4$ do not have this generator. The Cayley graph of group $\mathbb{Z}_2 \times\mathbb{Z}_2 \times \mathbb{Z}_2$ with any inverse closed subsets of cardinality $3$ is isomorphic to $\Gamma_G$ and $Cay(\{a,b^2,ab^3 \},D_8)\simeq \Gamma_G.$
		\end{itemize}
		\item[Case 2.] Let $s_1^{-1}=s_3$ and $O(s_2)=2.$ We show that $s_1s_3 \not\in E(\Gamma_G).$ If not, $s_3=s_1s_2$ and so $s_1=s_3s_2$. Hence, $n=O(s_1)=4$. This is contradiction by $n \ge 8$. Thus $N(s_2)=\{e,x=s_1s_2,y=s_3s_2\}$.\\
		Let $O(s_1)=3$. Then $s_3=s_1^2$ and $s_1s_3 \in E(\Gamma_G)$. If $x$ and $s_1$ are adjacent, then $s_1s_2=x=s_2s_1$. On the other hand, $s_3y^{-1}=s_2$.  So $s_3y \in E(\Gamma_G)$ and so $s_3s_2=y=s_2s_3$. Also $xy^{-1}=s_3$, i.e. $x,y \in E(\Gamma_G)$. This means $n=6$ which is impossible. The same argument applies when $s_3\in N(x)$ and $x \in N(y)$. Hence, $|N[S]|=8$. Since $G$ has some elements of order 2 and 3 and also $|V^-|=2$, $n=12$. The Cayley graph of group $A_4=\langle \{ a,b: a^2=b^3=(ab)^3=e\} \rangle$ with inverse closed subset $S=\{a,b, b^2\}$ is isomorphic to $\Gamma_G$. (Figure \ref{FigureCay(A4)})\\
		If $O(s_1)\neq 3$, then the induced subgraph on $S$ has at least one edge. If $N(s_1)=N(s_2)=\{ e,x,y\}$, then $x=s_1s_2$ and $y=s_1^{-1}s_2$. On the other hand, $ys_1^{-1}=s_1$ and $xs_1^{-1}=s_2$. Thus $s_2=s_1^3$ and $O(s_1)=6$. So $n=6$ which is a contradiction. With similar argument, we can see $N(s_1) \neq N(s_3)$. Hence, for each $i \neq j$, $N(s_i) \neq N(s_j)$. If $\cap_{i=1}^3 N(s_i)=\{e,x\}$, then $xs_2 \in \{s_1,s_1^{-1}\}$. Without loss of generality, suppose that $xs_2=s_1$. So $xs_1^{-1}=s_2$ and $s_2s_1=x=s_1s_2$. Also since $x \in N(s_1^{-1})$, so $xs_1\in \{s_2,s_1^{-1}\}$. If $xs_1=s_2$, then $s_2s_1^{-1}=s_2s_1$ and if $xs_1=s_1^{-1}$, then $s_2=s_1^3$. Both of them are impossible. Now we have four  cases likewise Case 1. Because of Remarks \ref{D10Z10} and \ref{D12,Z12}, Cases (ii) and (iii) can not be happened. 
		\begin{itemize}
			\item[(i)] With similar argument we see that $n \in \{10, 12\}$. The groups $A_4$ and $T=\langle \{a,b: a^4=b^3=e, \, a^{-1}ba=b^{-1} \}\rangle $ and $A_4=\langle \{ a,b: a^2=b^3=(ab)^3=e\} \rangle$ do not have such this generator but for other groups $\Gamma_G$ is isomorphic to following Cayley graphs:\\
			$ Cay\left(\{\bar{2}, \bar{5}, \bar{8}\}:\mathbb{Z}_{10}\right)$, $Cay\left(\{a,b,b^4\}:D_{10}\right)$, $Cay\left(\{a,b^3,ab\}:D_{12}\right)$,\\
			$ Cay\left(\{(\bar{0},\bar{1}),(\bar{1}, \bar{3}),(\bar{0}, \bar{5})\}:\mathbb{Z}_2 \times \mathbb{Z}_6\right)$, $Cay\left(\{\bar{1},  \bar{6}, \bar{11}\}:\mathbb{Z}_{12}\right).$
			\item[(iv)] In this case $|N(S)|=4$. Because of properties of the groups of order $10$ and $12$, $n \not\in \{10,12\}$ and also $G \not\simeq Q_8$. So $G \simeq D_8, \mathbb{Z}_2\times \mathbb{Z}_4$.  For any generator of these groups ($S=\{s_1, s_2, s_1^{-1}\}$), $Cay(S:G)\simeq \Gamma_G$ and this completes the proof.
		\end{itemize}

	\end{itemize}
\end{proof}

\begin{figure}
	\centering
	\begin{tikzpicture}[
	vertex_style/.style={circle,inner sep=0pt,minimum size=0.12cm, fill opacity=0.9, fill=black},
	edge_style/.style={line width=1pt, gray}]
	
	
	\begin{scope}[rotate=90]
	\foreach \x/\y in {20/1,70/2,110/3,160/4,200/5,250/6,290/7,340/8}{
		\node[vertex_style] (\y) at (canvas polar cs: radius=2cm,angle=\x){};
	}
	
	\foreach \x/\y in {90/9,180/10,270/11,360/12}
	\node[vertex_style] (\y) at (canvas polar cs: radius=1cm,angle=\x){};
	
	\end{scope}
	\foreach \x/\y in {1/2,2/3,3/4,4/5,5/6,6/7,7/8,8/1}{
		\draw[edge_style] (\x) -- (\y);
	}
	\foreach \x/\y in {9/11,10/12}{
		\draw[edge_style] (\x) -- (\y);
	}
	\foreach \x/\y in {1/12,12/8,9/2,3/9,4/10,5/10,6/11,7/11}{
		\draw[edge_style] (\x) -- (\y);
	}
	
	\end{tikzpicture}
	\caption{The cubic Cayley graph $Cay(S: A_4).$}
	
	\label{FigureCay(A4)}
\end{figure}
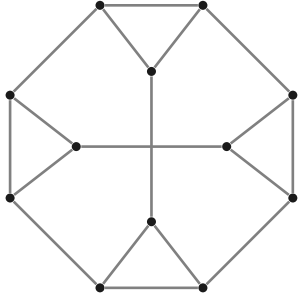

\end{document}